\documentclass[11pt,reqno]{amsart}

\usepackage[T1]{fontenc}
\usepackage{lmodern}
\usepackage{amsmath,amssymb,amsthm,mathtools}
\usepackage[margin=1.08in]{geometry}
\usepackage{microtype}
\usepackage[colorlinks=true,linkcolor=blue,citecolor=blue,urlcolor=blue]{hyperref}
\hypersetup{
  pdftitle={ Infinite Cascades for the Cubic Nonlinear Half-Wave Equation on the Line},
  pdfauthor={Xi Chen}
}

\numberwithin{equation}{section}
\newtheorem{theorem}{Theorem}[section]
\newtheorem{proposition}[theorem]{Proposition}
\newtheorem{lemma}[theorem]{Lemma}

\theoremstyle{remark}
\newtheorem{remark}[theorem]{Remark}

\newcommand{\R}{\mathbb R}
\newcommand{\N}{\mathbb N}
\newcommand{\X}{\mathcal X}

\newcommand{\eps}{\varepsilon}
\newcommand{\half}{\frac12}

\newcommand{\norm}[1]{\left\lVert #1\right\rVert}
\newcommand{\abs}[1]{\left\lvert #1\right\rvert}
\newcommand{\ip}[2]{\left\langle #1,#2\right\rangle}
\newcommand{\flow}[1]{\Phi^{#1}}
\newcommand{\trans}[1]{\tau_{#1}}

\title[Infinite Cascades for the cubic nonlinear half-wave equation]
{Infinite Cascades for the Cubic Nonlinear Half-Wave Equation on the Line}
\author{Xi Chen}
\address{Department of Mathematics and Computer Science, University of Basel, Spiegelgasse 1, 4051 Basel, Switzerland}
\email{xi01.chen@unibas.ch}

\begin{document}

\begin{abstract}
We construct infinite-time Sobolev cascades for the cubic nonlinear half-wave equation
on the real line.  For $\half<s\leq1$ and
$0\leq\kappa<2s(2s-1)/(4s-1)$, spatial gluing of permanently saturated
two-soliton packets gives a dense set of strictly mass-subthreshold
data for which $\|u(t)\|_{H^s}/(1+t)^\kappa\to\infty$ in the focusing case.  A second construction
replicates finite-burst Szeg\H{o} packets without rescaling.  It yields a dense $G_\delta$ subset of $H^\infty(\R)$ whose solutions perform
logarithmic cascades in every $H^s$ with $s>\half$ in the focusing case and in the defocusing case.
\end{abstract}

\maketitle

\section{Introduction}

\subsection{Background}

Questions about the growth of high Sobolev norms ask whether a solution can
transfer energy to increasingly short scales even when its lower-order norms
remain controlled by conservation laws.  An influential early contribution
was Bourgain's work on long-time Sobolev bounds \cite{BourgainGrowth}.  More
recently, Maspero and Murgante found small-to-large finite-time growth for a
one-dimensional quasilinear NLS with sublinear dispersion
\cite{MasperoMurgante}.

Genuine unbounded trajectories are known for several related models.  Hani
constructed such trajectories for a family of periodic Hamiltonian cubic
equations, including the resonant cubic NLS, rather than for the standard
cubic NLS itself \cite{Hani}.  Hani--Pausader--Tzvetkov--Visciglia proved
modified scattering for small data to a resonant system on
$\mathbb R\times\mathbb T^d$, where
$\mathbb T=\mathbb R/(2\pi\mathbb Z)$, for $1\leq d\leq4$ and, as an application,
obtained focusing and defocusing solutions with unbounded high Sobolev norms
for every $d\geq2$ \cite{HPTV}.
A result more closely related to the half-wave equation considered here is due
to Xu, who studied
\[
 (i\partial_t+\partial_x^2-|D_y|)U=|U|^2U
 \quad\text{on }\mathbb R_x\times\mathbb T_y.
\]
Here $|D_y|$ is the multiplier by the absolute $y$--frequency, and
\[
 \|U\|_{L_x^2H_y^s}
 :=\left(\int_{\mathbb R}\|U(x,\cdot)\|_{H_y^s}^2\,\mathrm dx\right)^{1/2}.
\]
She constructed a class of small solutions for which
$\|U(t)\|_{L_x^2H_y^s}$ is unbounded for every $s>\half$
\cite{XuWaveguide}.

In \cite{ChenPlane}, the author considered the corresponding equation on the plane,
\[
 (i\partial_t+\partial_x^2-|D_y|)U=|U|^2U,
 \qquad (x,y)\in\mathbb R_x\times\mathbb R_y.
\]
He constructed a class of infinite-time blow-up solutions satisfying
\[
 0<\liminf_{t\to\infty}
 \frac{\|U(t)\|_{L_x^2H_y^1}}{1+\log t}
 \leq\limsup_{t\to\infty}
 \frac{\|U(t)\|_{L_x^2H_y^1}}{1+\log t}<\infty.
\]

The cubic Szeg\H{o} equation was first introduced by G\'erard and Grellier, who found
its Lax pair and completely integrable structure
\cite{GerardGrellierSzego}.  Their later nonlinear Fourier transform on the
circle gives a detailed description of the flow.  In particular, they proved
that a dense $G_\delta$ set of smooth data has super-polynomial $H^s$
excursions, for every $s>\half$, along one sequence of times, while the
corresponding solutions return to their initial data along another
\cite{GerardGrellierHankel}. More recently, Gérard and Pushnitski proved the genericity of turbulent solutions of the
cubic Szeg\H o equation on the real line \cite{GerardPushnitskiInverse}.

On the real line, Pocovnicu proved soliton resolution for strongly generic
rational data \cite[Theorem~1.10]{PocovnicuExplicit}.  For nongeneric
rank-two data she obtained only partial soliton resolution and, in particular,
an orbit for which
\[
 \|W(t)\|_{H^s}\geq c_s|t|^{2s-1},
 \qquad |t|\gg1,\quad s>\half.
\]
This construction lies in the exceptional spectral regime where
$W_0=W(0)$ and the square of the Hankel operator
$H_{W_0}h:=\Pi_+(W_0\overline h)$ has a double eigenvalue; here $\Pi_+$ denotes
the projection onto nonnegative Fourier frequencies
\cite[Proposition~1.6, Theorem~1.11, and
Corollary~1.12]{PocovnicuExplicit}.  The cubic Szeg\H{o}
equation also governs the leading resonant dynamics of the cubic nonlinear half-wave
equation.  G\'erard and Grellier established this connection on the circle
\cite{GerardGrellierEffective}.  On the real line, Pocovnicu used a long-time
approximation to construct small solutions of the defocusing cubic nonlinear half-wave
equation whose relative $H^s$ growth at suitable finite times tends to
infinity as the size of the initial datum tends to zero
\cite[Theorem~1.2 and Corollary~1.4]{PocovnicuApprox}.  This approximation is
valid on a long but finite time
interval.  For the focusing equation,
G\'erard--Lenzmann--Pocovnicu--Rapha\"el
found a different mechanism: a two-soliton undergoes a transient turbulent
regime, after which its $H^1$ norm remains permanently large
\cite[Theorem~1.2]{GLPR}.

In this paper, we work with the cubic nonlinear half-wave equation
\begin{equation}\label{eq:hw}
 i\partial_tu-|D|u=\mu|u|^2u,
 \qquad u(0)=u_0,
 \qquad \mu\in\{+1,-1\},
\end{equation}
where
\[
 \widehat f(\xi):=(2\pi)^{-1/2}\int_{\mathbb R}e^{-ix\xi}f(x)\,\mathrm dx,
 \qquad D=-i\partial_x,
 \qquad \widehat{|D|f}(\xi)=|\xi|\widehat f(\xi).
\]
The sign $\mu=+1$ denotes the defocusing case and $\mu=-1$ denotes the focusing case.  We denote the flow by
$\flow\mu_tu_0$ and use the mass-critical symmetry
\begin{equation}\label{eq:scale}
 S_\lambda f(x)=\lambda^{1/2}f(\lambda x),
 \qquad
 \flow\mu_t(S_\lambda f)=S_\lambda\flow\mu_{\lambda t}f.
\end{equation}
The time scaling in \eqref{eq:scale} follows directly from the equation; the
linear dispersion has order one.

We use the notation
\[
 \langle x\rangle=(1+|x|^2)^{1/2},
\]
and, unless variables are displayed explicitly, all function spaces are over
$\mathbb R$.  For $1\leq p\leq\infty$ we abbreviate
\begin{equation}\label{eq:norm-conventions}
 \|f\|_p:=\|f\|_{L^p(\mathbb R)},\qquad
 \|f\|_{H^s}^2:=\int_{\mathbb R}\langle\xi\rangle^{2s}|\widehat f(\xi)|^2
 \,\mathrm d\xi,\qquad
 \|f\|_{\dot H^s}^2:=\int_{\mathbb R}|\xi|^{2s}|\widehat f(\xi)|^2
 \,\mathrm d\xi.
\end{equation}
The notation $\langle f,g\rangle_{H^s}$ refers to the inner product associated
with this $H^s$ norm, $\mathbf 1_E$ is the indicator of $E$, and
$\mathbb N=\{1,2,\ldots\}$.  We write $A\lesssim_\Theta B$ when
$A\leq C_\Theta B$, and $A\asymp_\Theta B$ when both inequalities hold;
$O_\Theta$ and $o_\Theta$ have the analogous meanings.  Subscripts are omitted
when the dependence is harmless.

Let $Q$ be the positive even ground state solving
$|D|Q+Q-Q^3=0$.  For $s>\half$ and $0<q\leq\norm Q_2$, set
\begin{equation}\label{eq:spaces}
 \X^+_s=H^s(\R),
 \qquad
 \X^-_{s,q}=\{f\in H^s(\R):\norm f_2<q\}.
\end{equation}
The Cauchy problem is locally well-posed in $H^s$ for $s\geq\half$, and the
flow is locally Lipschitz on bounded subsets when $s>\half$; we use the
formulation recalled in \cite{ChoffrutPocovnicu}.  The defocusing energy gives
global existence.  In the focusing case, the sharp Gagliardo--Nirenberg
inequality
\[
 \|f\|_4^4\leq \frac{2}{\|Q\|_2^2}
 \|f\|_2^2\,\||D|^{1/2}f\|_2^2
\]
and conservation of mass and energy give global existence whenever
$\|f\|_2<\|Q\|_2$; see \cite{FrankLenzmann,KLR,GLPR}. See also finite-time blow-up results for the focusing cubic half-wave equation in \cite{KLR,CaoSuZhang,KimKwonParkBW,KimKwonParkLogLog,ParkNegativeEnergy}.  Thus every datum in
$\X^-_{s,q}$ is strictly below the threshold, even when
$q=\|Q\|_2$.

\subsection{Results and ideas}

Two features of the known finite-burst constructions must be addressed in
order to obtain a smooth infinite cascade.  A sequence of bursts need not
control the orbit between event times, while the high-frequency scaling that
makes one packet small in $H^s$ makes it large in every stronger Sobolev norm.

For the focusing equation, the construction in \cite{GLPR} provides a packet whose
$H^s$ norm remains large after its activation time, rather than only at one
event time.  Write
\begin{equation}\label{eq:rho-star}
 \rho_s^*:=\frac{2s(2s-1)}{4s-1},
 \qquad \half<s\leq1.
\end{equation}

\begin{theorem}[Dense full-tail focusing cascades]\label{thm:limit}
Let $\half<s\leq1$, $0<q\leq\norm Q_2$, and
$0\leq\kappa<\rho_s^*$.  Then the set
\begin{equation}\label{eq:limit-set}
 \left\{f\in\X^-_{s,q}:
 \lim_{t\to\infty}
 \frac{\norm{\flow-_tf}_{H^s}}{(1+t)^\kappa}=+\infty\right\}
\end{equation}
is dense in $\X^-_{s,q}$.  Equivalently, every nonempty relatively open
subset of $\X^-_{s,q}$ contains such a datum.  In particular, at $\kappa=0$
one has the genuine limit $\norm{\flow-_tf}_{H^s}\to\infty$.
\end{theorem}

For $s=1$, formula \eqref{eq:rho-star} gives $\rho_1^*=2/3$.  The conclusion is
stronger in time than a limsup statement, although we prove density rather
than residuality.  The proof begins with a prescribed sequence of activation
times and places the corresponding permanent packets progressively farther
apart.  Spatial decoupling is used only on finite intervals, but these
intervals together cover the entire future time axis.  For this reason we use
a direct construction.  The full-tail condition is not naturally expressed
as an open condition, so the Baire-category argument used for the second
theorem is not suitable here.

For smooth data we instead keep Pocovnicu's packets at their original spatial
scale and replicate them finitely many times.  By almost orthogonality, both
the norm of the initial sum and that of its profile at the chosen later time
grow like the square root of the number of copies.  Because the output of each
packet carries an additional logarithmic factor, the number of copies can be
chosen so that the cluster is small in any prescribed Sobolev norm while its
later $H^s$ norm is large.  Put
\begin{equation}\label{eq:hinfty}
 H^\infty(\R):=\bigcap_{m=0}^\infty H^m(\R)
\end{equation}
with its standard Fr\'echet topology, and let
\begin{equation}\label{eq:hinfty-spaces}
 \X^+_\infty=H^\infty(\R),
 \qquad
 \X^-_{\infty,q}=\{f\in H^\infty(\R):\norm f_2<q\}.
\end{equation}
For $s>\half$, define
\begin{equation}\label{eq:alpha}
 \alpha_s:=\frac{2s-1}{4s-1}.
\end{equation}

\begin{theorem}[$H^\infty$ logarithmic cascades]\label{thm:hinfty}
Fix $s>\half$ and let
$h:[0,\infty)\to(0,\infty)$ satisfy
\begin{equation}\label{eq:gauge}
 h(t)=o\bigl((\log(2+t))^{\alpha_s}\bigr).
\end{equation}
For the defocusing sign, the set
\begin{equation}\label{eq:hinfty-fixed}
 \left\{f\in\X^+_\infty:
 \limsup_{t\to\infty}
 \frac{\norm{\flow+_tf}_{H^s}}{h(t)}=+\infty\right\}
\end{equation}
is a dense $G_\delta$ subset of $\X^+_\infty$.  For every
$0<q\leq\norm Q_2$, the analogous focusing set is a dense $G_\delta$ subset
of $\X^-_{\infty,q}$.

Moreover, for either sign there is a dense $G_\delta$ set, relative to the
same phase space, such that every datum $f$ in that set satisfies,
simultaneously for all real $r>\half$ and $0\leq\gamma<\alpha_r$,
\begin{equation}\label{eq:simultaneous}
 \limsup_{t\to\infty}
 \frac{\norm{\flow\mu_tf}_{H^r}}
 { (\log(2+t))^\gamma}=+\infty.
\end{equation}
The data may alternatively be obtained by a direct diagonal gliding-hump
construction in every nonempty Fr\'echet-open set.
\end{theorem}

The two conclusions have different strengths.  Theorem~\ref{thm:limit} gives a
genuine limit but only for focusing data in a fixed $H^s$, with
$\half<s\leq1$.  Theorem~\ref{thm:hinfty} gives one datum which cascades in all
$H^s$, $s>\half$, but only along a sequence of times.  Our argument does not
produce a Schwartz datum: the translations needed for decoupling create large
spatial moments.  This is a limitation of the construction, not a
nonexistence theorem.

Hani's work \cite{Hani} provides a useful methodological point of comparison. He pastes a turbulent solution, supported on a carefully chosen Fourier set, to a background solution with disjoint Fourier support; a combinatorial noninteraction condition makes this pasting exact. Long-time strong instability is then converted by a Baire-category argument into a residual family of unbounded Sobolev orbits \cite{Hani}. Our smooth-data construction follows the same broad instability-to-genericity principle, but the decoupling mechanism is different. For the cubic nonlinear half-wave equation on \(\mathbb R\), packets are separated by large translations in physical space and decouple only asymptotically on prescribed finite time intervals. In particular, we do not use an exact invariant Fourier-support decomposition of the type available in Hani's periodic models.

The Baire-category argument itself is therefore not the main novelty of this paper. What is new is the construction of packets for the full cubic nonlinear half-wave flow that can be added to arbitrary backgrounds without losing the required long-time information. In the focusing argument, the permanently saturated two-soliton packets of G\'erard--Lenzmann--Pocovnicu--Rapha\"el \cite{GLPR} are scheduled so that successive finite decoupling windows cover the entire future time axis. This gives a dense full-tail result with a quantified polynomial rate, rather than only unbounded growth along a sequence of times. For smooth data, we first replicate finitely many unscaled Pocovnicu packets. This avoids the loss in stronger Sobolev norms caused by high-frequency rescaling and makes the Baire argument work in the Fr\'echet space \(H^\infty(\mathbb R)\), producing, for either sign, simultaneous logarithmic cascades in every \(H^s\), \(s>\frac12\), while keeping the focusing data strictly below the mass threshold. Thus the connection with \cite{Hani} lies in the pasting/Baire framework; the packet mechanisms, the decoupling argument, and the full-tail and Fr\'echet-space conclusions are specific to the present work.
\subsection{Organization of the paper}
Section~\ref{sec:decoupling} establishes the finite-time decoupling tools used
in both constructions.  Section~\ref{sec:persistent-focusing} develops the
persistent focusing packets and proves Theorem~\ref{thm:limit}.
Section~\ref{sec:smooth-data} constructs the replicated smooth packets and
carries out the Fr\'echet-space argument for Theorem~\ref{thm:hinfty}.

\section{Finite-time decoupling tools}\label{sec:decoupling}

For $y\in\R$, let $(\trans yf)(x)=f(x-y)$.  We use the following standard
consequences of the algebra property of $H^s(\R)$ for $s>\half$.

\begin{lemma}[Separated flows]\label{lem:decouple}
Let $s>\half$, $f,g\in H^s(\R)$, and $T<\infty$.  Assume that the solutions
from $f$, $g$, and $f+\trans yg$ exist on $[0,T]$ for all sufficiently large
$|y|$.  Then
\begin{equation}\label{eq:decouple}
 \sup_{0\leq t\leq T}
 \norm{\flow\mu_t(f+\trans yg)-\flow\mu_tf-
 \trans y(\flow\mu_tg)}_{H^s}\longrightarrow0
\end{equation}
as $|y|\to\infty$.
\end{lemma}

\begin{proof}
The orbit segments of $u(t)=\flow\mu_tf$ and $v(t)=\flow\mu_tg$ are compact
in $H^s$.  Uniformly on $[0,T]$, products containing one untranslated factor
from the orbit of $u$ and one translated factor from the orbit of $v$ tend to
zero in $H^s$.  This follows first for fixed factors by approximation with
compactly supported smooth functions, and then uniformly for the compact
orbit segments by finite nets.  Consequently,
\[
 \sup_{0\leq t\leq T}
 \norm{|u+\trans yv|^2(u+\trans yv)-|u|^2u-
 \trans y(|v|^2v)}_{H^s}\longrightarrow0.
\]
Duhamel's formula, the local Lipschitz estimate for the cubic map on $H^s$,
and Gronwall's inequality prove \eqref{eq:decouple}.  The sign does not enter
the estimate.  In the focusing case the subthreshold assumption guarantees
the required existence.
\end{proof}

\begin{lemma}[Compact-uniform translation orthogonality]\label{lem:orthog}
If $K,L\subset H^s(\R)$ are compact, then
\begin{equation}\label{eq:orthog}
 \sup_{a\in K,\,b\in L}
 \abs{\ip a{\trans yb}_{H^s}}\longrightarrow0
 \qquad (|y|\to\infty).
\end{equation}
In particular, for fixed $a,b\in H^s$,
$\norm{a+\trans yb}_{H^s}^2\to\norm a_{H^s}^2+\norm b_{H^s}^2$.
\end{lemma}

\begin{proof}
For fixed $a,b$, the inner product is the Fourier transform at $y$ of
an $L^1$ function, so the Riemann--Lebesgue lemma applies.  Finite nets for
$K$ and $L$ give the uniform statement.
\end{proof}

For every $\varepsilon>0$, the two lemmas give the lower bound
\begin{equation}\label{eq:max-lower}
 \norm{\flow\mu_t(f+\trans yg)}_{H^s}
 \geq\max\{\norm{\flow\mu_tf}_{H^s},
 \norm{\flow\mu_tg}_{H^s}\}-\varepsilon.
\end{equation}
It holds uniformly on any prescribed finite interval once $|y|$ is large
enough (depending on $\varepsilon$).

\section{The persistent focusing construction}
\label{sec:persistent-focusing}

\subsection{Persistent focusing packets}

The construction begins with the solution supplied by
\cite[Theorem~1.2]{GLPR}, whose $H^1$ norm becomes large and remains so.  That
paper uses the convention
$iU_t+|D|U=|U|^2U$; complex conjugation converts it to the focusing sign in
\eqref{eq:hw}, without reversing time.

Fix $0<\delta\ll1$.  For all sufficiently small $\eta>0$, the solution
$U_\eta$ is defined on $[T_{\rm in},\infty)$, where
\begin{equation}\label{eq:times}
 T_{\rm in}=\eta^{-2\delta},
 \qquad T^-=\delta\eta^{-1}.
\end{equation}
Its conserved mass and $L^2$ norm satisfy
\begin{equation}\label{eq:glpr-mass}
 \norm{U_\eta(t)}_2^2\asymp\eta,
 \qquad \norm{U_\eta(t)}_2\asymp\eta^{1/2}.
\end{equation}
Its growth law on
$[T_{\rm in},T^-]$ is
\begin{equation}\label{eq:glpr-growth}
 \norm{U_\eta(t)}_{H^1}^2
 =\frac{t^2}{\eta}(1+O(\sqrt\delta)),
\end{equation}
and its saturated regime satisfies
\begin{equation}\label{eq:glpr-saturation}
 \norm{U_\eta(t)}_{H^1}^2
 =\eta^{-3}e^{O(1/\delta)},
 \qquad t\geq T^-.
\end{equation}
In particular, evaluating \eqref{eq:glpr-growth} at $T_{\rm in}$ gives
\begin{equation}\label{eq:glpr-initial-h1}
 \norm{U_\eta(T_{\rm in})}_{H^1}
 \lesssim_\delta\eta^{-1/2-2\delta}.
\end{equation}

\begin{remark}\label{rem:glpr-normalization}
Evaluating the growth law \cite[(1.10)]{GLPR} at
$T_{\rm in}=\eta^{-2\delta}$ gives the exponent in
\eqref{eq:glpr-initial-h1}.  The same exponent follows from the profile:
$1-\beta_2(T_{\rm in})\asymp\eta/T_{\rm in}^2
=\eta^{1+4\delta}$, the rescaled $\dot H^1$ norm of $Q_{\beta_2}$ is
uniformly nonzero, and the exact-solution remainder is negligible.  The
separate initial-derivative line in the statement of \cite[Theorem~1.2]{GLPR}
prints $\eta^{-1-2\delta}$ for the squared norm, whereas the displayed growth
law and the profile both give $\eta^{-1-4\delta}$.  We use the latter,
internally consistent value.
\end{remark}

We first derive from the two-soliton decomposition a lower bound at
intermediate Sobolev regularity.  This form of the estimate is not stated
explicitly in \cite{GLPR}.

We recall the notation in \cite{GLPR} used below.  The solitary-wave profile $Q_\beta$
is normalized by
\[
 \frac{|D|-\beta D}{1-\beta}Q_\beta+Q_\beta=|Q_\beta|^2Q_\beta.
\]
For the two bubbles, $\lambda_j(t)$, $x_j(t)$, $\beta_j(t)$, and
$\gamma_j(t)$ denote respectively the scale, center, velocity, and phase
parameters.  For a fixed large approximation order $N_0$, set
\[
 \kappa_j=\lambda_j(1-\beta_j),\qquad
 y_j=\frac{x-x_j}{\kappa_j},\qquad
 R=\frac{x_2-x_1}{\kappa_1},\qquad
 b=\frac{1-\beta_2}{1-\beta_1},
\]
and let $V_j^{(N_0)}$ be the order-$N_0$ corrected profile in the variable
$y_j$.  Its physical rescaling is
\[
 \mathcal R_j(t,x)=\lambda_j(t)^{-1/2}
 V_j^{(N_0)}(t,y_j)e^{i\gamma_j(t)}.
\]
Time arguments will be suppressed when no confusion is possible.

\begin{lemma}[Sobolev size of the saturated bubble]\label{lem:glpr-hs}
Let $\half<s\leq1$ and put $a=s-\half$.  After fixing $\delta$ and taking
$\eta$ sufficiently small,
\begin{align}
 \norm{U_\eta(T_{\rm in})}_{H^s}
 &\lesssim_{s,\delta}
 \eta^{-a-2s\delta},\label{eq:initial-hs}\\
 \inf_{t\geq T^-}\norm{U_\eta(t)}_{\dot H^s}
 &\gtrsim_{s,\delta}\eta^{-3a}.
 \label{eq:saturated-hs}
\end{align}
\end{lemma}

\begin{proof}
Since $\norm{U_\eta(T_{\rm in})}_2\asymp\eta^{1/2}$,
the interpolation inequality between $L^2$ and $H^1$, together with
\eqref{eq:glpr-initial-h1}, gives
\[
 \norm{U_\eta(T_{\rm in})}_{H^s}
 \lesssim
 \eta^{(1-s)/2}
 \eta^{-s(1/2+2\delta)}
 =\eta^{-a-2s\delta}.
\]

For the lower bound, write the approximate solution in
\cite[Proposition~4.6]{GLPR} as
\[
 \Phi^{(N_0)}(t)=\mathcal R_1(t)+\mathcal R_2(t),\qquad
 \kappa_j(t):=\lambda_j(t)(1-\beta_j(t)),
\]
where $\mathcal R_j$ is the $j$th rescaled correction profile.  In the
saturated regime,
\[
 \kappa_2(t)\asymp_\delta\eta^3,\qquad
 \kappa_1(t)\asymp\eta,\qquad R(t)\asymp_\delta t,
 \qquad t\geq T^-.
\]
These parameter estimates are recorded in
\cite[Theorem~1.2 and Lemma~4.15]{GLPR}.
The profile convergence in \cite[Proposition~2.3]{GLPR} gives an annulus
$I\Subset(0,\infty)$ and a cutoff $\chi\in C_c^\infty(I)$ for which
$\chi(D)Q_{\beta_2}$ has uniformly nonzero $L^2$ norm.  Define the rescaled
projection $P_I$ and its physical-scale versions by
\[
 \widehat{P_IF}(\xi)=\chi(\xi)\widehat F(\xi),\qquad
 \widehat{P_{j,t}F}(\xi)=\chi(\kappa_j(t)\xi)\widehat F(\xi),
 \quad j=1,2.
\]
The finite expansion in \cite[Proposition~4.6]{GLPR} expresses
$V_j^{(N_0)}-Q_{\beta_j}$ as correction profiles carrying at least one factor
$R^{-1}$.  More precisely, the weighted estimate in the proof of
\cite[Corollary~4.7]{GLPR} gives
\[
 |V_j^{(N_0)}(y)-Q_{\beta_j}(y)|
 \lesssim_\delta
 \frac{R^{-1}}{\langle y\rangle(1+(1-\beta_j)|y|)}.
\]
Integrating this bound gives the concrete consequence
\begin{equation}\label{eq:profile-correction-l2}
 \norm{V_j^{(N_0)}-Q_{\beta_j}}_{L_y^2}
 \lesssim_\delta R^{-1},\qquad j=1,2.
\end{equation}
The annulus used below is fixed in the rescaled $y$--frequency.  Hence
\eqref{eq:profile-correction-l2}, followed by the physical rescaling, says
that the correction has annular $H^s$ contribution $O_\delta(R^{-1})$
relative to the corresponding leading bubble.  In particular, $P_{2,t}$
projects to $|\xi|\asymp\kappa_2(t)^{-1}$.  Since
$\lambda_j(t)\asymp1$, the rescaling used here may be written explicitly as
\begin{equation}\label{eq:annular-rescaling}
 \mathcal B_{j,F}(x)
 =\lambda_j^{-1/2}F\left(\frac{x-x_j}{\kappa_j}\right),
 \qquad
 \norm{P_{j,t}\mathcal B_{j,F}}_{\dot H^s}
 =\lambda_j^{-1/2}\kappa_j^{1/2-s}
 \norm{P_I F}_{\dot H^s},
\end{equation}
In particular, a rescaled correction $F$ contributes at most
$C_I\lambda_j^{-1/2}\kappa_j^{1/2-s}\norm F_2$ on that annulus.  Since
$R(t)^{-1}=O_\delta(\eta)$ on this time range, the second profile satisfies
\begin{equation}\label{eq:second-annulus}
 \norm{P_{2,t}\mathcal R_2(t)}_{\dot H^s}
 \geq c_{s,\delta}\kappa_2(t)^{1/2-s}
 \gtrsim_{s,\delta}\eta^{-3a}.
\end{equation}
On the same annulus, split the first profile into the physical rescaling of
the bare bubble and its correction,
$\mathcal R_1=\mathcal R_1^Q+\mathcal R_1^{\rm corr}$.  The first-profile
coefficient in \cite[Proposition~4.6]{GLPR} yields the stronger estimate
\[
 \norm{V_1^{(N_0)}-Q_{\beta_1}}_{L_y^2}
 \lesssim_\delta
 \frac{b}{R(1+(1-\beta_1)R)}
 \lesssim_\delta\eta^3,
 \qquad t\geq T^-,
\]
because $b\asymp_\delta\eta^2$, $1-\beta_1\asymp\eta$, and
$R\gtrsim_\delta\eta^{-1}$ in the saturated regime.  The uniform rescaled
$H^1$ bound for $Q_{\beta_1}$ follows from
\cite[Proposition~2.3]{GLPR}.  Since $\kappa_1\asymp\eta$, physical rescaling
and Bernstein's inequality therefore give
\begin{align}
 \norm{P_{2,t}\mathcal R_1^Q(t)}_{\dot H^s}
 &\lesssim \kappa_2(t)^{1-s}
 \norm{\mathcal R_1^Q(t)}_{\dot H^1}
 \lesssim_\delta\eta^{3(1-s)}\eta^{-1/2}
 =O_\delta(\eta)\eta^{-3a},\label{eq:first-annulus-bare}\\
 \norm{P_{2,t}\mathcal R_1^{\rm corr}(t)}_{\dot H^s}
 &\lesssim \kappa_2(t)^{-s}\kappa_1(t)^{1/2}\eta^3
 \lesssim_\delta\eta^2\eta^{-3a}.
 \label{eq:first-annulus-correction}
\end{align}
Finally, estimate \cite[(5.82)]{GLPR}, followed by the compactness passage in
the proof of \cite[Theorem~1.2]{GLPR}, gives the quantitative exact-solution
remainder
\begin{equation}\label{eq:error-annulus}
 \norm{P_{2,t}(U_\eta(t)-\Phi^{(N_0)}(t))}_{\dot H^s}
 \lesssim \kappa_2(t)^{1-s}t^{-N_0/10}
 =o_{s,\delta}(\eta^{-3a}),
\end{equation}
after taking the fixed approximation order $N_0$ large.  Equations
\eqref{eq:second-annulus},
\eqref{eq:first-annulus-bare}--\eqref{eq:error-annulus} exclude cancellation and
prove \eqref{eq:saturated-hs} uniformly for $t\geq T^-$.  At $s=1$ one may
alternatively use the full $H^1$ saturation directly.
\end{proof}

Recall that $a=s-\half$, and set
\begin{equation}\label{eq:abc}
 e_{s,\delta}:=a+2s\delta,
 \qquad
 b_{s,\delta}:=2a-2s\delta,
 \qquad
 \gamma_{s,\delta}:=1+\frac as+2\delta.
\end{equation}
Choose $\delta<a/s$, so that $b_{s,\delta}>0$, and put
\begin{equation}\label{eq:rho-delta}
 \rho_{s,\delta}:=\frac{b_{s,\delta}}{\gamma_{s,\delta}}.
\end{equation}
Then
\begin{equation}\label{eq:rho-limit}
 \rho_{s,\delta}\longrightarrow
 \frac{2s(2s-1)}{4s-1}=\rho_s^*
 \qquad (\delta\downarrow0).
\end{equation}

\begin{proposition}[Persistent quantitative packets]\label{prop:persistent}
Fix $\half<s\leq1$ and $\delta>0$ as above.  There exist constants
$c,C>0$ such that the following holds.  Given $0<d\leq1$ and every
sufficiently large target time $B$ (with threshold allowed to depend on
$d$), there are a focusing datum $g\in H^1$
and an activation time $A\in[B,CB]$ satisfying
\begin{align}
 \norm g_{H^s}&<d,\label{eq:persist-small}\\
 \inf_{t\geq A}\norm{\flow-_tg}_{H^s}
 &\geq c\,d^{1+\rho_{s,\delta}/s}B^{\rho_{s,\delta}}.
 \label{eq:persist-plateau}
\end{align}
The $L^2$ norm of $g$ tends to zero as $B\to\infty$.
\end{proposition}

\begin{proof}
Let $V_\eta(t)=\overline{U_\eta(T_{\rm in}+t)}$ and choose
$\lambda>0$ by
\begin{equation}\label{eq:persistent-scale}
 \lambda^s=c_0d\eta^{e_{s,\delta}},
\end{equation}
where $c_0$ is a sufficiently small fixed constant.  Put
$g=S_\lambda V_\eta(0)$.  For the relevant parameters $0<\lambda\leq1$, and
\begin{equation}\label{eq:inhomogeneous-small-scale}
 \norm{S_\lambda f}_{H^s}^2
 \leq \norm f_2^2+\lambda^{2s}\norm f_{\dot H^s}^2.
\end{equation}
Lemma~\ref{lem:glpr-hs}, \eqref{eq:glpr-mass}, and
\eqref{eq:persistent-scale} therefore give \eqref{eq:persist-small}, after
requiring $\eta^{1/2}\leq c_0d$ and decreasing $c_0$ once.  This requirement
is met by taking the target time $B$ sufficiently large, with a threshold
depending on $d$.
The activation time and saturated height obey
\begin{align}
 A&=\frac{T^--T_{\rm in}}\lambda
 \asymp_{s,\delta}d^{-1/s}\eta^{-\gamma_{s,\delta}},
 \label{eq:persist-time}\\
 \inf_{t\geq A}\norm{\flow-_tg}_{H^s}
 &\gtrsim_{s,\delta}
 \lambda^s\eta^{-3a}
 \asymp d\eta^{-b_{s,\delta}}.
 \label{eq:persist-height}
\end{align}
Scaling preserves the $L^2$ norm.  Since the right-hand side of
\eqref{eq:persist-time} varies continuously and tends to infinity as
$\eta\downarrow0$, choose $\eta$ so that $A\in[B,CB]$.  Eliminating $\eta$
between \eqref{eq:persist-time} and \eqref{eq:persist-height} gives
\eqref{eq:persist-plateau}.
\end{proof}

\subsection{Gluing persistent packets and proof of Theorem~\ref{thm:limit}}

Fix $\half<s\leq1$ and $\kappa<\rho_s^*$, and choose $\delta$ so small that
$\kappa<\rho:=\rho_{s,\delta}$.  If $\kappa>0$, choose
\begin{equation}\label{eq:theta}
 1<\theta<\frac\rho\kappa;
\end{equation}
for $\kappa=0$, take any $\theta>1$.  Let $f_*\in\X^-_{s,q}$ and let
$\sigma>0$.  Choose $d_*>0$ sufficiently small, and set $d_j=d_*2^{-j}$, so
that
\begin{equation}\label{eq:budgets}
 \sum_jd_j<\min\{\sigma,q-\norm{f_*}_2\}.
\end{equation}
Choose $B_0>1$ later, and let
\begin{equation}\label{eq:targets}
 B_j=B_0^{\theta^{j-1}}.
\end{equation}
By increasing $B_0$, Proposition~\ref{prop:persistent} supplies packets
$g_j$ and activations $A_j\in[B_j,CB_j]$ such that
\begin{equation}\label{eq:scheduled}
 \norm{g_j}_{H^s}<d_j,
 \qquad
 \inf_{t\geq A_j}\norm{\flow-_tg_j}_{H^s}
 \geq(2j+2)(1+A_{j+1})^\kappa.
\end{equation}
Indeed, the quotient of the lower bound in
\eqref{eq:persist-plateau} by $(1+CB_{j+1})^\kappa$ contains
$B_j^{\rho-\theta\kappa}$, whose positive power beats the geometric loss in
$d_j$ and the factor $j$.  More explicitly, solving
\eqref{eq:persist-time} for $\eta$ shows that the required smallness
condition on $\eta$, together with $\eta^{1/2}<d_j$, imposes at most a
fixed-power threshold in $d_j^{-1}$.  The tower
$B_j=B_0^{\theta^{j-1}}$ dominates these thresholds uniformly after increasing
$B_0$.  We may also ensure
$CB_j<B_{j+1}$ and hence $A_{j+1}>A_j+1$.

Preselect all packets and activation times before choosing any translations.
Set $F_0=f_*$.  Once $F_{j-1}$ has been defined, choose $y_j$ so large that
Lemmas~\ref{lem:decouple}--\ref{lem:orthog} give, uniformly for
$0\leq t\leq A_{j+1}$,
\begin{equation}\label{eq:stage-lower}
 \norm{\flow-_t(F_{j-1}+\trans{y_j}g_j)}_{H^s}
 \geq
 \max\{\norm{\flow-_tF_{j-1}}_{H^s},
 \norm{\flow-_tg_j}_{H^s}\}-\eps_j,
\end{equation}
where $\eps_j>0$ and $\sum_j\eps_j<1$.  Define
$F_j=F_{j-1}+\trans{y_j}g_j$.  All partial sums are focusing subthreshold by
\eqref{eq:budgets}, so every application of finite-time decoupling is valid.

The series
\begin{equation}\label{eq:F-limit}
 F=f_*+\sum_{j=1}^\infty\trans{y_j}g_j
\end{equation}
converges absolutely in $H^s$, belongs to $\X^-_{s,q}$, and satisfies
$\norm{F-f_*}_{H^s}<\sigma$.  If $t\in[A_j,A_{j+1}]$, first apply
\eqref{eq:stage-lower} at stage $j$, then at every later stage.  Since the
horizon $A_{k+1}$ at stage $k>j$ contains this interval,
\[
 \norm{\flow-_tF_n}_{H^s}
 \geq(2j+2)(1+A_{j+1})^\kappa-
 \sum_{k=j}^n\eps_k.
\]
For fixed $j$, continuous dependence and $F_n\to F$ give uniform convergence
of the flows on $[0,A_{j+1}]$.  Letting $n\to\infty$ yields
\begin{equation}\label{eq:window-final}
 \inf_{A_j\leq t\leq A_{j+1}}
 \frac{\norm{\flow-_tF}_{H^s}}{(1+t)^\kappa}
 \geq j
\end{equation}
after increasing $B_0$ once more.  The windows cover $[A_1,\infty)$, so
\eqref{eq:window-final} proves Theorem~\ref{thm:limit}.

The order of choices matters here: $A_{j+1}$ is known before $y_j$ is chosen.
Thus the decoupling horizon at stage $j$ covers the entire interval assigned
to packet $j$; no separation distance depends on an event time which is still
to be chosen.

\section{Smooth-data cascades}\label{sec:smooth-data}

\subsection{Replicated smooth packets}

We now turn to the smooth-data construction, which applies to both signs.  Fix
$s>\half$ and choose a sufficiently small parameter
$\delta_0=\delta_0(s)>0$ in Pocovnicu's approximation theorem
\cite{PocovnicuApprox}.  With this convention, all constants below may depend
on $s$ and $\delta_0$.  Pocovnicu's construction applies to the defocusing
sign, and the observation in \cite[footnote~1]{GLPR} gives the corresponding
focusing family.  For the positive sign one may take
\[
 W_0^+(x)=\frac1{x+i}-\frac2{x+2i},
\]
the datum in \cite[Proposition~1.3]{PocovnicuApprox}; for the negative sign we
take $W_0^-(x)=\overline{W_0^+(-x)}$.  Thus $W_0^\mu\in H^\infty(\R)$, and the
corresponding solutions satisfy
\begin{equation}\label{eq:base-family}
 p_\eps^\mu=\eps W_0^\mu,
 \qquad
 t_\eps\asymp_s\eps^{-2}(\log(1/\eps))^{1/(4s-1)},
\end{equation}
such that, with $L_\eps=\log(1/\eps)$,
\begin{equation}\label{eq:base-signal}
 \norm{\flow\mu_{t_\eps}p_\eps^\mu}_{H^s}
 \geq c_s\eps L_\eps^{\alpha_s}.
\end{equation}
Since $\delta_0$ is fixed throughout, its contribution inside the logarithms
is absorbed into the constants.  The exponent $\alpha_s$ in
\eqref{eq:base-signal} is the one proved in
\cite[proof of Corollary~1.4, especially (2.20)--(2.21)]{PocovnicuApprox}.
The printed statement of Corollary~1.4 has exponent
$(4s-2)/(4s-1)$, whereas its proof gives $(2s-1)/(4s-1)$; we use the proved
exponent.  In
particular,
\begin{equation}\label{eq:base-time-log}
 \log t_\eps=2L_\eps+O_s(\log L_\eps).
\end{equation}

\cite[Lemma~2.2]{PocovnicuApprox} proves the oscillatory estimates required
above for every $s>\half$.  In particular, no separate low-frequency
argument is needed in the range $\half<s<1$.

For completeness, we explain the change of sign.  Define the antilinear
isometry
\[
 (Jf)(x):=\overline{f(-x)}.
\]
Then
\[
 \widehat{Jf}(\xi)=\overline{\widehat f(\xi)},
\]
so $J$ preserves positive Fourier support and every Sobolev norm.  If
\[
 i\partial_tW=\Pi_+(|W|^2W),
\]
then
\[
 i\partial_t(JW)=-\Pi_+(|JW|^2JW).
\]
Thus the growing positive-sign Szeg\H{o} orbit gives the corresponding
negative-sign orbit.  The normal-form proof of
\cite[Theorem~1.2]{PocovnicuApprox} is unchanged after reversing the cubic
sign: the resonant and oscillatory cubic terms change sign, while all the
estimates used in the argument remain the same.  Consequently,
\eqref{eq:base-family}--\eqref{eq:base-signal} hold for both signs.  This
focusing modification is also noted in \cite[footnote~1]{GLPR}.
\begin{proposition}[Smooth cluster packet]\label{prop:cluster}
Fix $s>\half$, an integer $M\geq s$, a sign $\mu$, and a gauge $h$ satisfying
\eqref{eq:gauge}.  Given $d,\zeta,A,L_0>0$, there exist
$G\in C_c^\infty(\R)$ and $T>L_0$ such that
\begin{equation}\label{eq:cluster-properties}
 \norm G_{H^M}<d,
 \qquad
 \norm G_2<\zeta,
 \qquad
 \norm{\flow\mu_TG}_{H^s}>A(h(T)+1).
\end{equation}
For the focusing sign, $G$ and all partial sums used in its construction can
be chosen strictly subthreshold.
\end{proposition}

\begin{proof}
Let
\[
 R_\eps=A(h(t_\eps)+1),
 \qquad
 P_\eps=\norm{\flow\mu_{t_\eps}p_\eps^\mu}_{H^s},
\]
and choose
\begin{equation}\label{eq:N-cluster}
 N_\eps=\left\lceil
 \left(\frac{8R_\eps}{P_\eps}\right)^2
 \right\rceil.
\end{equation}
Set
\[
 C_M:=\|W_0^\mu\|_{H^M},\qquad C_{L^2}:=\|W_0^\mu\|_2.
\]
Since $\norm{p_\eps^\mu}_{H^M}=C_M\eps$, equations
\eqref{eq:base-signal}--\eqref{eq:base-time-log} and the gauge assumption give
\begin{equation}\label{eq:cluster-cost}
 \sqrt{N_\eps}\,\eps
 \lesssim_{s,M,A}
 \frac{h(t_\eps)+1}{L_\eps^{\alpha_s}}+\eps
 \longrightarrow0.
\end{equation}

For fixed $\eps$, abbreviate $N=N_\eps$, $p=p_\eps^\mu$, and
$w=\flow\mu_{t_\eps}p$.  Choose translations recursively and set
\[
 G_k=\sum_{\ell=1}^k\trans{y_\ell}p,\qquad
 Z_k=\sum_{\ell=1}^k\trans{y_\ell}w,\qquad G_0=Z_0=0.
\]
At step $k$, Lemmas~\ref{lem:decouple}--\ref{lem:orthog} allow $y_k$ to be
chosen so that
\begin{align}
 \norm{\flow\mu_{t_\eps}G_k-\flow\mu_{t_\eps}G_{k-1}
 -\trans{y_k}w}_{H^s}
 &<\frac{R_\eps}{2N},\label{eq:cluster-step-flow}\\
 \abs{\ip{G_{k-1}}{\trans{y_k}p}_{H^M}}
 &<\frac12 C_M^2\eps^2,\label{eq:cluster-step-initial}\\
 \abs{\ip{Z_{k-1}}{\trans{y_k}w}_{H^s}}
 &<\frac18P_\eps^2.\label{eq:cluster-step-terminal}
\end{align}
For the focusing sign we simultaneously impose
$|\ip{G_{k-1}}{\trans{y_k}p}_{L^2}|<C_{L^2}^2\eps^2/2$.  The resulting estimate
\[
 \norm{G_k}_2\leq\sqrt{2k}\,C_{L^2}\eps
 \leq2C_{L^2}\sqrt N\,\eps=o(1)
\]
keeps every partial aggregate strictly subthreshold, so that the next use of
Lemma~\ref{lem:decouple} is legitimate.

Summing \eqref{eq:cluster-step-flow}, and expanding the squares in
\eqref{eq:cluster-step-initial}--\eqref{eq:cluster-step-terminal}, shows that,
for
\begin{equation}\label{eq:G-cluster}
 G_\eps=G_N,
\end{equation}
one has
\begin{align}
 \norm{G_\eps}_{H^M}
 &\leq2C_M\sqrt{N_\eps}\eps,
 \label{eq:G-small}\\
 \norm{\flow\mu_{t_\eps}G_\eps-
 \sum_{k=1}^{N_\eps}\trans{y_k}
 \flow\mu_{t_\eps}p_\eps^\mu}_{H^s}
 &<R_\eps,\label{eq:G-flow-error}\\
 \norm{\sum_{k=1}^{N_\eps}\trans{y_k}
 \flow\mu_{t_\eps}p_\eps^\mu}_{H^s}
 &\geq\frac12\sqrt{N_\eps}P_\eps.
 \label{eq:G-output-orthog}
\end{align}
The last two inequalities and \eqref{eq:N-cluster} yield a terminal norm
larger than $3R_\eps$.  Equation \eqref{eq:cluster-cost} makes both the
$H^M$ and $L^2$ norms arbitrarily small, and also allows $t_\eps>L_0$.

Thus $G_\eps\in H^\infty$ has all the desired strict inequalities.  We may
now approximate this finite sum in $H^M$ by a function in $C_c^\infty$.
Fixed-time continuous dependence in $H^M$, with $M\geq s$, preserves the
terminal lower bound, the initial estimate, and the focusing mass inequality.
\end{proof}

\subsection{The Fr\'echet argument and proof of Theorem~\ref{thm:hinfty}}

The spaces in \eqref{eq:hinfty-spaces} are Baire spaces; in particular,
$\X^-_{\infty,q}$ is an open subset of the Fr\'echet space $H^\infty$.
Persistence of regularity and continuous dependence in each integer $H^M$,
$M\geq1$, show that, on every fixed finite time interval, the flow is
continuous in the Fr\'echet topology (the $L^2$ seminorm is controlled by
$H^1$).  For $m,N\in\N$, set
\begin{equation}\label{eq:open-events}
 \mathcal O_{m,N}=
 \bigcup_{t>N}\left\{f:
 \norm{\flow\mu_tf}_{H^s}>m h(t)\right\}.
\end{equation}
Here $f$ ranges over $\X^+_\infty$ for the defocusing sign and over
$\X^-_{\infty,q}$ for the focusing sign.
Each $\mathcal O_{m,N}$ is relatively open.  It is dense as well.  Indeed, a
basic Fr\'echet neighborhood of a datum $f$ constrains only finitely many
Sobolev seminorms, hence contains an $H^M$ ball around $f$ for some sufficiently
large integer $M\geq s$.  Invoke Proposition~\ref{prop:cluster} with
$A=4m$ to obtain a cluster which is small in that $H^M$ norm and whose
future $H^s$ signal has a strict margin over $m h(T)$.  Translate the entire
cluster far from $f$ and apply
Lemmas~\ref{lem:decouple}--\ref{lem:orthog} at time $T$.  In the focusing
case choose its $L^2$ norm below the remaining mass margin.  This proves
density.

Consequently,
\begin{equation}\label{eq:residual-fixed}
 \bigcap_{m,N=1}^\infty\mathcal O_{m,N}
\end{equation}
is the dense $G_\delta$ set in the first part of
Theorem~\ref{thm:hinfty}.

For the simultaneous statement, intersect \eqref{eq:residual-fixed} over all
rational $\sigma>\half$ and all rational
$0\leq\gamma<\alpha_\sigma$, with
$h(t)=(\log(2+t))^\gamma$.  Let $r>\half$ and
$0\leq\gamma<\alpha_r$ be real.  The function $r\mapsto\alpha_r$ is
continuous and strictly increasing, so one may choose rational
$\sigma\in(\half,r)$ and rational $\gamma'$ such that
\[
 \gamma<\gamma'<\alpha_\sigma.
\]
The continuous embedding $H^r\hookrightarrow H^\sigma$ then transfers the
limsup conclusion for $(\sigma,\gamma')$ to the one for $(r,\gamma)$.
This proves \eqref{eq:simultaneous}.

The same density statement also has a direct diagonal construction.  Start from a
datum in a prescribed basic neighborhood of order $M_0$ and enumerate the
rational pairs $(\sigma_j,\gamma_j)$ so that every pair occurs infinitely
often.  Set $T_0=0$.  At stage $j$ choose
\[
 M_j\geq\max\{M_0,j,\lceil\sigma_1\rceil,\ldots,
 \lceil\sigma_j\rceil\}.
\]
Assume that the first $j-1$ event inequalities and corresponding common
continuity radii $r_1,\ldots,r_{j-1}$ have already been fixed.  Choose the
$j$th cluster so that
\[
 \norm{G_j}_{H^{M_j}}
 <2^{-j}\min\{1,r_1,\ldots,r_{j-1}\},
\]
as well as an overall summable budget for the original neighborhood and, in
the focusing case, for the remaining $L^2$ norm.  Apply
Proposition~\ref{prop:cluster} with target factor $A=8j$, target gauge
$(\log(2+t))^{\gamma_j}$, and $T_j>T_{j-1}+1$; then translate the whole
cluster far enough to create the new event by
Lemmas~\ref{lem:decouple}--\ref{lem:orthog}, with errors small enough that the
new normalized event is larger than $4j$.  Fixed-time continuity at the
finitely many times $T_1,\ldots,T_j$ now supplies $r_j>0$, chosen as a common
radius in the finite product of the norms
$H^{\sigma_1},\ldots,H^{\sigma_j}$, such that all established inequalities
survive and the newest normalized event remains larger than $2j$.

After decreasing the numerical budgets once, the displayed size condition
gives $\sum_{k>j}\norm{G_k}_{H^{\sigma_i}}<r_j$ for every $i\leq j$.
Consequently all earlier events survive in the limit.  For each fixed integer
$m$, the terms with $k\geq m$ are summable in
$H^m$ because $M_k\geq k$.  The series therefore converges in every $H^m$
and hence in $H^\infty$, while every rational pair has event factors tending
to infinity.  This gives a datum in every prescribed Fr\'echet neighborhood.

\end{document}